\documentclass[12pt,a4paper]{amsart}
\usepackage{url}
\usepackage{amsmath,amsthm,amsfonts,amssymb,latexsym}

\newtheorem{theorem}{Theorem}[section]
\newtheorem{proposition}[theorem]{Proposition}
\newtheorem{lemma}[theorem]{Lemma}

\theoremstyle{definition}
\newtheorem{definition}[theorem]{Definition}

\newtheorem{question}[theorem]{Question}

\newcommand{\U}{\mathcal U}
\newcommand{\w}{\omega}

\newcommand{\IP}{\mathbb P}

\newcommand{\Op}{\mathcal{O}\/}
\newcommand{\B}{\mathcal{B}}

\newcommand{\K}{\mathcal{K}}

\newcommand{\F}{\mathcal{F}}

\newcommand{\V}{\mathcal{V}}

\newcommand{\bigvid}{\hat{\ \ }}

\newcommand{\uhr}{\upharpoonright}

\newcommand{\name}[1]{\dot{#1}}
\newcommand{\la}{\langle}
\newcommand{\ra}{\rangle}

\newcommand{\Split}{\mathit{Split}}
\newcommand{\Lev}{\mathit{Lev}}

\newcommand{\forces}{\Vdash}

\newcommand{\hot}{\mathfrak}
\newcommand{\W}{\mathcal W}

\newcommand{\nothing}[1]{}

\title[Menger spaces in the Miller model]{Products of Menger spaces in the Miller model}

\author{Lyubomyr Zdomskyy}

\address{Kurt G\"odel Research Center for Mathematical Logic\footnote{Current address: Institute of Discrete Mathematics and Geometry, Technische Universit\"at Wien,
Wiedner Hauptstra\ss e 8--10,
1040 Vienna, Austria },
University of Vienna, W\"ahringer Stra\ss e 25, A-1090 Wien,
Austria.}
\email{lzdomsky@gmail.com}
\urladdr{http://www.logic.univie.ac.at/\~{}lzdomsky/}


\thanks{
The  author would
like to thank  the Austrian Science Fund FWF (Grants I 1209-N25 and I 2374-N35)
 for generous support for this research.}

\begin{document}
\maketitle

\begin{abstract}
We prove that in the Miller model the Menger property
is preserved by finite products of metrizable spaces. This answers several open
questions and gives another instance of the interplay between classical forcing 
posets with fusion and combinatorial covering properties in topology.
\end{abstract}

\section{Introduction}

A topological space
$X$ has the  \emph{Menger} property (or, alternatively, is a Menger space)
 if for every sequence $\la \U_n : n\in\omega\ra$
of open covers of $X$ there exists a sequence $\la \V_n : n\in\omega \ra$ such that
 $\V_n\in[\U_n]^{<\w}$ for all  $n$ and  $X=\bigcup\{\cup \V_n:n\in\omega\}$.
 This property was introduced by  Hurewicz, and the current name
(the Menger property) is used because Hurewicz
proved in   \cite{Hur25} that for metrizable spaces his property is equivalent to
one   considered by Menger in \cite{Men24}.
Each $\sigma$-compact space has obviously the 
Menger property, and the latter implies lindel\"ofness (that is, every open 
cover has a countable subcover). The Menger property is the weakest one among the
so-called \emph{selection principles} or \emph{combinatorial covering properties},
see, e.g.,  \cite{BarSheTsa03, Koc04, SakSch13, Sch03, Tsa06} for detailed introductions to the topic. 
 Menger  spaces have recently found applications in such
areas as forcing \cite{ChoRepZdo15}, Ramsey theory in algebra
\cite{Tsa??},  combinatorics of discrete subspaces \cite{Aur10}, and
Tukey relations between hyperspaces of compacts \cite{GarMedZdo??}.

In this paper we proceed our investigation of the interplay between 
posets with fusion and selection principles initiated in \cite{RepZdo??}.
More precisely, we concentrate on the question whether the Menger property is preserved 
by finite products. For general topological spaces the answer negative:
  In ZFC there are two
normal spaces $X,Y$ with a covering property much stronger than the
Menger one such that $X\times Y$ does not have the  Lindel\"of
property, see \cite[\S 3]{Tod95}. However, the above situation
becomes impossible if we restrict our attention to metrizable
spaces. This  case, on which we  concentrate in the sequel,
 turned out to be  sensitive to the ambient set-theoretic universe.
Indeed, by \cite[Theorem~3.2]{MilTsaZso16} under CH there exist
 $X,Y\subset\mathbb R$ which have the Menger property
(in fact, they have the  strongest combinatorial covering property considered thus far), 
whose product $X\times Y$ is not
 Menger. There are many results of this kind where CH is relaxed to
an equality between cardinal characteristics, see, e.g., 
\cite{Bab09, COC2, RepZdo10, TsaSch02}. Surprisingly, there are also inequalities 
between cardinal characteristics which imply that the Menger property is not
productive even for sets of reals, see \cite{SzeTsa??}. 
The following theorem, which is the main result of our paper,
shows that an additional set-theoretic assumption in all these results was indeed necessary.

\begin{theorem} \label{main}
In the Miller model, 
the product of any two Menger spaces
is Menger provided that it is Lindel\"of. In particular, in this model the product of any
two Menger metrizable spaces is Menger.
\end{theorem}

Theorem~\ref{main} answers 
\cite[Problem~7.9(2)]{SzeTsa??},
\cite[Problem~8.4]{Tsa03_04} (restated as \cite[Problem~4.11]{Tsa_lecce}), 
  and
\cite[Problem 6.7]{Tsa07} 
in the affirmative; 
 implies that the affirmative answer to
 \cite[Problem~II.2.8]{Ark92},
\cite[Problem~3.9]{BelBonMatTka08}, 
and
 \cite[Problem~2]{COC2}
(restated as \cite[Problem~3.2]{Tsa06} and \cite[Problem~2.1]{Tsa07})
is consistent; 
implies that the negative answer to 
and \cite[Problem~II.2.7]{Ark92},  
\cite[Problem~8.9]{BerTam15}, and \cite[Problems~1,2,3]{Zdo06}
is consistent; and answers \cite[Problem~7]{SakSch13}
in the negative.

By the \emph{Miller model} we mean a generic extension of a ground  model of GCH
with respect to the iteration of length $\w_2$ with countable support of the Miller forcing,
see the next section for its definition. This model has been first considered by 
Miller in \cite{Mil84} and since then found numerous applications, see \cite{Bla10}
and references therein. The Miller forcing is similar to the Laver one introduced in 
\cite{Lav76}, the main difference being that the splitting is allowed to occur less
often. The main technical part of the proof of Theorem~\ref{main}
is Lemma~\ref{miller_like_laver} which is an analog of   
\cite[Lemma~14]{Lav76}. The latter one was the key ingredient in the proof that 
all strong measure zero sets of reals are countable in the Laver model given in 
\cite{Lav76}, the arguably most quotable combinatorial feature of the 
Laver model.

As we shall see in Section~\ref{proofs}, a big part of the proof of Theorem~\ref{main}
requires only the inequality $\hot u<\hot g$ which holds in the Miller model. 
However, we do not know the answer to the following

\begin{question}
Is the Menger property preserved by finite products of metrizable spaces under $\hot u<\hot g$?
If yes, can $\hot u<\hot g$ be weakened to  
 the Filter Dichotomy, NCF, or $\hot u<\hot d$?
\end{question}

We refer the reader to \cite[\S~9]{Bla10} for corresponding definitions.

We assume that the reader is familiar with the basics of forcing.
The paper is essentially self-contained in the sense that we give all the definitions
needed to understand our proofs.

\section{Proofs} \label{proofs}

First we collect some basic properties of Menger spaces. In most of the survey articles on combinatorial covering properties 
these are assumed to be a kind of folklore and are not  stated (even without  proofs) at all, see, e.g., 
\cite{COC2, Tsa03_04,  Tsa06, Tsa07}. The second item of the following lemma is a partial case of \cite[Proposition~2.1]{Tsa_lecce}.
\begin{lemma}\label{basics_1}
\begin{enumerate}
\item If $Z$ is a Menger space and $X$ is a closed subspace of $Z$, then $X$ is Menger when considered with the subspace topology
inherited from $Z$.
\item If $\mathcal X$ is a countable family of Menger subspaces of a space $Z$, then $\bigcup\mathcal X$
is a Menger subspace of $Z$.
\item If $X$ is Menger and $K$ is compact, then $X\times K$ is Menger. Consequently,
if  $X$ is Menger and $Y$ is $\sigma$-compact, then $X\times Y$ is Menger.
In particular, if  $X$ is Menger and $Q$ is countable, then $X\times Q$ is Menger.
\end{enumerate}
\end{lemma}
\begin{proof}
1.  Let $\la \U_n:n\in\w\ra$ be a sequence of open covers of $X$. For every $U\in \U_n$ fix an open subset $W(U)$ of $Z$
such that $W(U)\cap X=U$, and set $\W_n=\{W(U):U \in \U_n\}\cup \{Z\setminus X\}$.
Applying the Menger property of $Z$ to the sequence $\la W_n:n \in \w\ra$ of open covers of $Z$ we can find a sequence
$\la \V_n:n \in \w\ra$ such that $\V_n$ is a finite subset of $\W_n$ and $\bigcup_{n\in\w}\bigcup \V_n=Z.$
Write each $\V_n$ in the form $\{W(U^n_i):i<k_n\}$, where $U^n_i \in \U_n$ and $k_n\in\w$,
and note that $X$ is covered by $\{U^n_i:n\in\w, i<k_n\}$. This proofs the Menger property of $X$.

2. Let $\la\U_n:n\in\w\ra$ be a sequence of covers of $\bigcup\mathcal X$
by sets open in $Z$. Let us write $\omega$ in the form $\bigcup_{X\in\mathcal X}I_X$
such that each $I_X$ is infinite and $I_{X_0}\cap I_{X_1}=\emptyset$ for any $X_0\neq X_1$ in $\mathcal X$.
Since $\mathcal X$ is a family of Menger spaces, for every $X\in\mathcal X$ there exists a sequence
$\la \V_n:n\in I_X\ra$ such that $\V_n\in [\U_n]^{<\w}$ for all $n\in I_X$ and
$X\subset\bigcup_{n\in I_X}\cup\V_n$. Then $\bigcup\mathcal X\subset\bigcup_{n\in\w}\cup\V_n$ which proves that
$\bigcup\mathcal X $ is Menger.

3. Let $\la\U_n:n\in\w\ra$ be a sequence of open covers of $ X\times K$.
Let us fix $x\in X$ and $n\in\w$. Since $\{x\}\times K$ is a compact subspace of $X\times K$, there exists 
a finite $\V_{x,n}\subset\U_n$ such that $\{x\}\times K\subset\cup\V_{x,n}$. 
By \cite[Lemma~3.1.15]{Eng89} there exists an open subset $O_{x,n}\ni x$ of $X$
such that $O_{x,n}\times K\subset \cup\V_{x,n}$.  Applying the Menger property of $X$
to the sequence $\la\Op_n:n\in\w\ra$ of open covers of $X$, where 
$\Op_n=\{O_{x,n}:x\in X\}$, we get a sequence $\la F_n:n\in\w\ra$ of finite subsets of
$X$ such that $X=\bigcup_{n\in\w}\bigcup_{x\in F_n}O_{x,n}$.
Set $\V_n=\bigcup_{x\in F_n}\V_{x,n}$ and note that 
$$ \bigcup_{x\in F_n}O_{x,n}\times K\subset \bigcup_{x\in F_n}\cup\V_{x,n}=\cup\V_n $$
for every $n\in\w$, 
and hence $X\times K=\bigcup_{n\in\w}\cup\V_n$. This completes our proof of the fact that $X\times K$
is Menger. 

Now suppose that $Y=\bigcup\K$ where $\K$ is a countable collection of compact subspaces of $Y$.
It follows from the above that $X\times K$ is Menger for any $K\in\mathcal K$, and hence $X\times Y$
is Menger by the second item being a countable union of its Menger subspaces.
\end{proof}

The proof of Theorem~\ref{main} is based on the fact that in the Miller model
spaces with the Menger property enjoy certain concentration properties 
defined below. Recall that a subset $R$ of a topological space 
$X$ is called a \emph{$G_{\w_1}$-set} if $R$ is an intersection of $\w_1$-many
open subsets of $X$.

\begin{definition}
 A topological space $X$ is called \emph{weakly $G_{\w_1}$-con\-cen\-tra\-ted}
(resp. \emph{weakly $\w G_{\w_1}$-concentrated}) if for every collection $\mathsf Q\subset [X]^\w$ which is cofinal with respect
to inclusion, and for every function $R:\mathsf Q\to\mathcal P(X)$
assigning to each $Q\in\mathsf Q$ a $G_{\w_1}$-set $R(Q)$ containing $Q$,
there exists $\mathsf Q_1\in [\mathsf Q]^{\w_1}$ such that $X\subset\bigcup_{Q\in\mathsf Q_1}R(Q)$
(resp. for every $Q\in [X]^\w$ there exists $Q_1\in\mathsf Q_1$ with the property
 $Q\subset R(Q_1)$).
\end{definition}

Let $A$ be a countable  set and $x,y\in\w^A$. As usually
$x\leq^* y$ means  that $\{a\in A: x(a)>y(a)\}$ is finite.
If $x(a)\leq y(a)$ for all $a\in A$, then we write $x\leq y$.
The smallest cardinality of a dominating with respect to
$\leq^*$ subset of $\w^\w$ is denoted by  $\hot d$.
The smallset cardinality of a family $\B\subset [\w]^\w$ generating 
an ultrafilter (i.e., such that $\{A:\exists B\in\B\:(B\subset A)\}$
is an ultrafilter) is denoted by $\hot u$. By \cite[Theorem~2]{BlaLaf89}
combined with the results of \cite{BlaShe89} the inequality 
$\w_1=\hot u<\hot g=\w_2$ holds in the Miller model,
 see \cite{BlaLaf89} or \cite{Bla10} 
for the definition of $\hot g$ as well as systematic treatment of cardinal characteristics of
reals. 

As the  following fact  established in \cite{MilTsaZso14} shows,
the inequality $\hot u<\hot g$ imposes  strong restrictions on the structure 
of Menger spaces.

\begin{lemma} \label{good_bound}
In the Miller model, for every Menger space $X\subset\mathcal P(\w)$
and a $G_\delta$-subset $G$ such that $X\subset G\subset\mathcal P(\w)$,
there exists a family $\mathcal K$ of compact subsets of
$G$ such that $|\mathcal K|=\w_1$ and $X\subset\bigcup\mathcal K$.

Consequently, in this model
for every Menger space $X\subset\mathcal P(\w)$
and continuous $f:X\to\w^\w$ there exists $F\in [\w^\w]^{\w_1}$
such that for every $x\in X$ there exists $f\in F$ with $f(x)\leq f.$ 
\end{lemma}
\begin{proof}
The first statement is \cite[Theorem~4.4]{MilTsaZso14} combined with the fact
that $\hot u=\w_1$ in the Miller model.  

Regarding the second statement, since the Menger property is preserved by continuous images and 
$\w^\w$ is homeomorphic to a $G_\delta$-subset of $\mathcal P(\w)$, there exists 
a family $\mathcal K$ of compact subsets of
$\w^\w$ such that $|\mathcal K|=\w_1$ and $X\subset\bigcup\mathcal K$.
For every $K\in\mathcal K$ there exists $f_K\in\w^\w$ such that
$y\leq f_K$ for all $y\in K$. It follows that the family $F=\{f_K:K\in\mathcal K\}$
is as required.
\end{proof}

By a \emph{Miller tree} we understand a subtree $T$ of $\w^{<\w}$ consisting of increasing finite sequences
such that the following conditions are satisfied:
\begin{itemize}
 \item Every $t\in T$ has an extension $s\in T$ which is splitting in $T$, i.e.,
there are more than one immediate successors of $s$ in $T$;
\item If $s$ is splitting in $T$, then it has infinitely many successors in $T$. 
\end{itemize}
 The \emph{Miller forcing} is the collection $\mathbb M$ of all  Miller trees ordered 
by inclusion, i.e.,
smaller trees carry more information about the generic. 
This poset has been introduced in \cite{Mil84} and since then found numerous applications
see, e.g., \cite{BlaShe89}. 
We denote by $\IP_\alpha$ an iteration of length $\alpha$ of the Miller forcing 
 with countable support. 
If $G$ is  $\IP_\beta$-generic and $\alpha<\beta$, then we denote the intersection
$G\cap\IP_\alpha$ by $G_\alpha$.   

For a Miller tree $T$ we shall denote by $\Split(T)$ the set of all splitting nodes
of $T$, and for some $t\in\Split(T)$ we denote the size of $\{s\in\Split(T):s\subsetneq t\}$
by $\Lev(t,T)$. For a node $t$ in a Miller tree $T$ we denote by $T_t$
the set $\{s\in T:s$ is compatible with $t\}$. It is clear that $T_t$ is also a Miller
tree. The \emph{stem} of a Miller tree $T$ is the (unique) $t\in\Split(T)$
such that $\Lev(t)=0$. We  denote the stem of $T$ by $T\la 0\ra$.
If $T_1\leq T_0 $ and $T_1\la 0\ra=T_0\la 0\ra$, then we  write 
$T_1\leq^0 T_0$.

The following lemma may be proved by an almost literal repetition of
the proof of \cite[Lemma~14]{Lav76}.

\begin{lemma} \label{miller_like_laver}
Let $\la\name{x}_i:i\in\w\ra$ be a sequence of $\IP_{\w_2}$-names for  reals
and $p\in\IP_{\w_2}$. Then there exists $p'\leq p$ such that
$p'(0)\leq^0 p(0)$, and a finite set of reals $U_s$ for each $s\in\Split(p'(0))$,
such that for each $\varepsilon>0$,  $s\in \Split(p'(0))$ with $\Lev(s,p'(0))=i$, $j\leq i$,
and for all but finitely many immediate successors $t$ of $s$ in $p'(0)$ we have
$$ (p'(0))_t\hat{\ \ }p'\uhr [1,\w_2)\forces\exists u\in U_s\: (|\name{x}_j-u|<\varepsilon).$$
\end{lemma}

A subset $C$ of $\w_2$ is called an
\emph{$\w_1$-club} if it is unbounded and for every $\alpha\in\w_2$ of cofinality $\w_1$,
if $C\cap\alpha$ is cofinal in $\alpha$ then $\alpha\in C$.

\begin{lemma} \label{miller}
 In the Miller model every Menger subspace of $\mathcal P(\w)$ is 
weakly $\w G_{\w_1}$-concentrated
(and hence also weakly $G_{\w_1}$-concentrated\footnote{The proof of Theorem~\ref{main}
requires only that Menger spaces are weakly $G_{\w_1}$-concentrated in the Miller
model. }).
\end{lemma}
\begin{proof}
We work in $V[G_{\w_2}]$, where $G_{\w_2}$ is $\IP_{\w_2}$-generic.
 Let us fix a Menger space $X\subset\mathcal P(\w)$,  consider a 
cofinal $\mathsf Q\subset [X]^\w$,  and let  $R(Q)
\supset Q$ be a $G_{\w_1}$-set for all $Q\in\mathsf Q$. 

 In the Miller model the
 Menger property is preserved by unions of  $\w_1$-many spaces, see 
\cite[Theorem~4]{Zdo05}  and \cite{BlaLaf89}
for the fact that $\hot g=\w_2$ in this model.
Combined with Lemma~\ref{basics_1}(1) this implies in particular that a complement $X\setminus R$ of an arbitrary $G_{\w_1}$-subset
$R\subset \mathcal P(\w)$ is Menger.
Therefore by Lemma~\ref{good_bound} and a standard 
closing off
argument (see, e.g., the proof of \cite[Lemma~5.10]{BlaShe87})
there exists an $\w_1$-club $C\subset \w_2$ such that for every $\alpha\in C$ 
the following condition is satisfied:
\begin{quote}
 $\mathsf Q\cap V[G_\alpha]$ is cofinal in 
$[X]^\w\cap V[G_\alpha]$, and for every continuous  $f$ from a subset of $\mathcal P(\w)$
into $\w^\w$ such that $f$ is coded in $V[G_\alpha]$,  
and every  $Q\in \mathsf Q\cap V[G_\alpha]$ such that
$X\setminus R(Q)\subset\mathrm{dom}(f)$, for every
$x\in X\setminus R(Q)$ there exists $b\in\w^\w\cap V[G_\alpha]$
such that $f(x)<b$.
\end{quote}
Let us fix $\alpha\in C$. 
We claim that $\mathsf Q\cap V[G_\alpha]$ has the required property. 
Suppose, contrary to our claim, there exists $p\in G_{\w_2}$ and a 
$\IP_{\w_2}$-name $\name{Q}_*$ 
such that $p$ forces ``$\name{Q}_* \in [\name{X}]^\w$ and
 $\name{Q}_*\not\subset \name{R}(
\name{Q})$ for any $\name{Q} \in [\name{X}]^\w\cap V[\Gamma_\alpha]$'',
where $\Gamma_\alpha$ is the standard $\IP_\alpha$-name for $\IP_\alpha$-generic filter.
There is no loss of generality\footnote{This follows from the fact that 
the remainder $\IP_{[\alpha,\w_2)}$ of $\IP_{\w_2}$ evaluated by $G_\alpha$, 
is forcing equivalent to
 $\IP_{\w_2}^{V[G_\alpha]}$ in  $V[G_\alpha]$. This can be proved in the same way as
\cite[Lemma~11]{Lav76}. }
in assuming that $\alpha=0$. Applying Lemma~\ref{miller_like_laver}
to a sequence $\la \name{q}_i:i\in\w\ra$ enumerating $\name{Q}_*$,
we get a condition $p'\leq p$ such that $p'(0)\leq^0 p(0)$, and a finite set $U_s$
of reals for every $s\in \Split(p'(0))$  such that for each $\varepsilon>0$
and each $s\in \Split(p'(0))$ with $\Lev(s, p'(0))=i$, for all but finitely many
immediate successors $t$ of $s$ in $p'(0)$ and all $j\leq i$ we have\footnote{Here we standardly identify $\mathcal P(\w)$
with the Cantor ternary subset of the unit interval, built by removing the middle thirds. This allows us
to speak about $|x-y|$ for elements $x,y\in\mathcal P(\w)$.}
\begin{equation} \label{eq1}
  p'(0)_t\bigvid p'\uhr[1,\w_2)\forces \exists u\in U_s\: (|\name{q}_j-u|<\varepsilon).
\end{equation}
Note that any condition stronger than  $p'$ will also satisfy condition~(\ref{eq1}),
because for $q_0,q_1\in\mathbb M$ the inequality $q_1\leq q_0$ implies
$\Lev(s,q_1)\leq \Lev(s,q_0)$ for all $s\in\Split(q_1)$.
Fix $Q\in\mathsf Q\cap V$ containing
$X\cap\bigcup\{U_s:s\in \Split(p'(0))\}$ and set $F=X\setminus R(Q)$.
It follows from the above that $p'\forces \name{Q}_*\not\subset \name{R}(Q)$.
By passing to a stronger condition, if necessary, we may additionally assume 
that $p'\forces \name{q}_j\not\in\name{R}(Q)$  for a given $j\in\w$. 

 Consider the map
$f:F\to\w^{\Split(p'(0))}$  defined as follows:
$$f(y)(s)=[1/\min\{|y-u|:u\in U_s\}]+1$$
 for\footnote{Here $[a]$ is the largest integer not exceeding $a$.} all $s\in \Split(p'(0))$
 and $y\in F$. Since $F$ is disjoint from $Q$,  $f$
is well-defined.
Since  $f$ is coded in $V$, 
there exist $p''\leq p'$ and $b\in\w^{\Split(p'(0))}\cap V$
such that $p''$ forces  $\name{f}(\name{q}_j)\leq b$. 
Without loss of generality we may additionally assume that 
$\Lev(p''(0)\la 0\ra, p'(0))\geq j$.
Letting $s''=p''(0)\la 0\ra$, we conclude  that
 $p''\Vdash \name{f}(\name{q}_j)(s'')\leq b(s'')$, which means that
$$ p''\Vdash     \min\{|\name{q}_j-u|:u\in U_{s''}\} \geq 1/b(s'').      $$
On the other hand, by our choice of $p'$,  $p''\leq p'$,
condition (\ref{eq1}), and \\ $\Lev(s'', p'(0))\geq j$, 
for all but finitely many
immediate successors $t$ of $s''$ in $p''(0)$ we have
$$  p''(0)_t\bigvid p''\uhr[1,\w_2)\forces \exists u\in U_{s''}\: |\name{q}_j-u|<1/b(s'') $$
which means
$ p''(0)_t\bigvid p''\uhr[1,\w_2)\forces \min\{|\name{q}_j-u|:u\in U_{s''}\} <1/b(s'') $
and thus leads to a contradiction.
\end{proof}

The next lemma relates the weak $G_{\w_1}$-concentration to products with Menger spaces.

\begin{lemma} \label{covering_g_delta}
In the Miller model, let $Y\subset \mathcal P(\w)$ be  Menger  and $Q\subset\mathcal P(\w)$
be countable. Then for every $G_{\w_1}$-subset $O$ of $\mathcal P(\w)^2$ containing
$Q\times Y$ there exists  a $G_{\w_1}$-subset $R\supset Q$ such that
$R\times Y\subset O$.
\end{lemma}
\begin{proof}
 Without loss of generality we shall assume that $O$ is open. Let us write
$Q$ in the form $\{q_n:n\in\w\}$ and set $O_n=\{z\in\mathcal P(\w):\la q_n,z\ra\in O\}\supset
Y$. By Lemma~\ref{good_bound} there exists a collection $\mathcal Z=\{Z_\alpha:\alpha\in\w_1\}$
 of compact subsets of $\bigcap_{n\in\w}O_n$ covering $Y$. It follows from the above that
$Q\times Z_\alpha\subset O$ for all $\alpha$. Applying \cite[Lemma~3.1.15]{Eng89}, for every $q\in Q$
and $\alpha<\w_1$ we get an open  $R_{q,\alpha}\subset\mathcal P(\w)$ containing $q$ such that
$R_{q,\alpha}\times Z_\alpha\subset O$. Then
$R_\alpha=\bigcup_{q\in Q}R_{q,\alpha}$ is an open neighbourhood of $Q$ in $\mathcal P(\w)$ with the property  $R_\alpha\times Z_\alpha\subset O$.
Letting $R=\bigcap_{\alpha<\w_1}R_\alpha$ we get that
$R\times Y\subset R\times\bigcup_{\alpha<\w_1}Z_\alpha\subset O$.
\end{proof}

As it was proved in \cite{BlaShe89}, in the Miller model there exists an ultrafilter $\F$
generated by $\w_1$-many sets, say $\{F_\alpha:\alpha\in\w_1\}$. There is a 
natural linear pre-order $\leq_\F$ on $\w^\w$ associated to $\F$ defined as follows:
$x\leq_\F y$ if and only if $\{n\in\w:x(n)\leq y(n)\}\in \F$. 
By \cite[Theorem~3.1]{BlaMil99}, in this model  for every $X\subset\w^\w$
of size $\w_1$ there exists $b\in\w^\w$ such that $x\leq_\F b$ for all $x\in X$.

\begin{lemma} \label{mod_f_in_miller}
In the Miller model, suppose that $\U_n=\{U^n_k:k\in\w\}$
is an open cover of a Menger space $X\subset\mathcal P(\w)$,
 for every $n\in\w$. Then there exists 
$b\in\w^\w$ such that $X\subset\bigcup_{n\in F_\alpha}\bigcup_{k\leq b(n)} U^n_k$
for all $\alpha$. (Equivalently, $\{n\in\w:x\in\bigcup_{k\leq b(n)}U^n_k\}\in \F$
for all $x\in X$.)
\end{lemma}
\begin{proof}
The equivalence of two statements follows from the equality $\F=\F^+$,
where for $\mathcal X\subset [\w]^\w$ we standardly denote by $\mathcal X^+$ 
the set $\{Y\subset\w:Y\cap X\neq\emptyset$ for all $X\in\mathcal X\}$.

To prove the second statement set  $G=\bigcap_{n\in\w}\bigcup \U_n$
and find  a collection $\K$ of compact subsets of $G$ such that
$|\K|=\w_1$ and $X\subset\bigcup\K\subset G$. This is possible by Lemma~\ref{good_bound}. 
For every $K\in\K$ find $b_K\in\w^\w$ such that $K\subset\bigcup_{k\leq b_K(n)}U^n_k$
for all $n\in\w$. Then any $b\in\w^\w$ such that $b_K\leq_\F b$ for all $K\in\K$
is easily seen to be as required.
\end{proof}

The second part of Theorem~\ref{main} is a direct consequence of Lemma~\ref{miller}
and the following

\begin{proposition} \label{main1}
 In the Miller model, let $Y \subset \mathcal P(\w)$ be a Menger space
and $X\subset \mathcal P(\w)$ be weakly $G_{\w_1}$-concentrated. Then $X\times Y$ is Menger.
\end{proposition}
\begin{proof}
Fix a sequence $\la \U_n:n\in\w\ra$ of  covers of $X\times Y$
by clopen subsets of $\mathcal P(\w)^2$. For every
 $Q\in [X]^\w $  
fix a sequence
$\la \W^{Q}_n:n\in \w\ra$  such that $\W^{Q}_n\in [\U_n]^{<\w}$ and
$Q\times Y\subset O_{Q,\alpha}$ for all $\alpha\in\w_1$, where 
$O_{Q,\alpha}=\bigcup_{n\in F_\alpha}\cup\W^{Q}_n$. (The latter is possible by Lemma~\ref{basics_1}(3) ensuring
that $Q\times Y$ is Menger, combined with
Lemma~\ref{mod_f_in_miller}.)
Letting $O_{Q}=\bigcap_{\alpha\in\w_1}O_{Q,\alpha}$ and using Lemma~\ref{covering_g_delta},
we can find a $G_{\w_1}$-subset $R(Q)\supset Q$ of $\mathcal P(\w)$
such that $R(Q)\times Y\subset O_Q$.
Since $X $ is  weakly $G_{\w_1}$-concentrated, 
there exists $\mathsf Q_1\in [[X]^\w]^{\w_1}$
such that $X\subset\bigcup_{Q\in\mathsf Q_1}R(Q)$. 
For every $Q\in\mathsf Q_1$  let us find
$b_{Q}\in\w^\w$ such that $\W^{Q}_n\subset\{U^n_k:k\leq b_{Q}(n)\}$
for all $n\in \w$, where $\U_n=\{U^n_k:k\in\w\}$ is an enumeration.
Let $b\in\w^\w$ be an upper bound of $\{b_{Q}:Q\in\mathsf Q_1\}$
with respect to $\leq_\F$. We claim that 
$X\times Y\subset\bigcup_{n\in\w}\bigcup_{k\leq b(n)}U^n_k$.
Indeed, fix $y\in Y$, $x\in X$, and find $Q\in \mathsf Q_1$ such that $x\in R(Q)$.
It follows that $\la x,y\ra \in O_{Q,\alpha}$ for all $\alpha\in\w_1$,
therefore for every $\alpha$ there exists $n\in F_\alpha$ with 
$\la x,y \ra\in \cup\W^Q_n\subset\bigcup_{k\leq b_Q(n)} U^n_k$,
and hence $F:=\{n\in\w: \la x,y\ra \in \bigcup_{k\leq b_Q(n)} U^n_k\}\in\F^+=\F$.
Then $\la x,y\ra \in \bigcup_{k\leq b(n)} U^n_k$ for all $n\in F\cap\{k:b_{Q}(k)\leq b(k)\}\in\F$,
which completes our proof.
\end{proof}

Note that Lemma~\ref{miller} together with Proposition~\ref{main1}
imply that in the Miller model, a subspace $X$ of $\mathcal P(\w)$ is Menger iff it is
weakly $G_{\w_1}$-concen\-tra\-ted iff it is
weakly $\w G_{\w_1}$-concen\-tra\-ted.

As we have already noticed above,  Lemma~\ref{miller} and Proposition~\ref{main1} imply Theorem~\ref{main}
for subspaces of $\mathcal P(\w)$. The general case of arbitrary Menger spaces can be reduced to 
subspaces of $\mathcal P(\w)$  in the same way as in the proof 
 of \cite[Theorem~1.1]{RepZdo??},
the only difference being that in some places ``Hurewicz'' should be replaced with ``Menger''.
However, we present this proof for the sake of completeness. We will need
 characterizations
of the  Menger property  obtained in \cite{Zdo05}.
Let  $u=\la U_n : n\in\omega\ra$ be a sequence of subsets of a set $X$.
For every $x\in X$ let  $I_s(x,u,X)=\{n\in\omega:x\in U_n\}$. If every
$I_s(x,u,X)$ is infinite (the collection of all such sequences $u$ will be denoted
by $\Lambda_s(X)$), then we shall denote by $\mathcal U_s(u,X)$
the smallest semifilter on $\omega$ containing all $I_s(x,u,X)$.
(Recall that a family $\F\subset[\w]^\w$ is called a \emph{semifilter} if for every
$F\in\F$ and $X^*\supset F$ we have $X\in\F$, where $F\subset^* X$
means $|F\setminus X|<\w$.)
By \cite[Theorem~3]{Zdo05},  a Lindel\"of topological space $X$  is Menger  if and only if
for every  $u \in\Lambda_s(X)$ consisting of open sets,
  the semifilter $\mathcal U_s(u,X)$ is Menger when considered with the subspace topology inherited from $\mathcal P(\w)$.
The proof given there also works if we consider only those
$\la U_n : n\in\omega\ra\in\Lambda_s(X)$ such that all $U_n$'s belong to a given base of
$X$.
\medskip

\noindent\textit{Proof of Theorem~\ref{main}.} \
Suppose that $X,Y$ are arbitrary Menger spaces such that $X\times Y$ is Lindel\"of
and fix  $w=\la U_n\times V_n :n\in\w\ra\in\Lambda_s(X\times Y)$
consisting of open sets.
Set  $u=\la U_n:n\in\w\ra$,  $v=\la V_n:n\in\w\ra$, and note that
$u\in\Lambda_s(X)$ and $v\in\Lambda_s(Y)$.
It is easy to see
that
$$\U_s(w,X\times Y)=\{A\cap B: A\in \U_s(u,X), B\in \U_s(v,Y)\},$$
and hence $\U_s(w,X\times Y)$ is a continuous image of
$\U_s(u,X)\times \U_s(v,Y)$. By \cite[Theorem~3]{Zdo05} both of latter ones
are Menger, considered as subspaces of $\mathcal P(\w)$, and hence
by Lemma~\ref{miller} and Proposition~\ref{main1} 
their product is  Menger as well. Thus
$\U_s(w,X\times Y)$ is Menger, being a continuous image of a Menger space.
It suffices to use \cite[Theorem~3]{Zdo05} again.
\hfill $\Box$
\medskip

\noindent \textbf{Acknowledgment.} We would like to thank Arnold Miller  for his comments
regarding Lemma~\ref{miller_like_laver}, and Marion Scheepers for the discussion of relations between Menger and  productively
Lindel\"of spaces. We thank the anonymous referee for a careful reading and many
suggestions which improved the presentation.

\end{document}